\documentclass[10pt]{amsart}

\usepackage{amsthm}
\usepackage{amssymb}
\usepackage{amsmath}
\usepackage[colorlinks=true, linkcolor=blue, citecolor=blue]{hyperref}

\newtheorem{theorem}{Theorem}[section]
\newtheorem{lemma}[theorem]{Lemma}
\newtheorem{corollary}[theorem]{Corollary}
\newtheorem{proposition}[theorem]{Proposition}
\theoremstyle{definition}
\newtheorem{example}[theorem]{Example}

\newcommand{\bs}{\backslash}

\title{The nucleus of a semisymmetric quasigroup}
\author{Andrew R.~Kozlik}
\address{Department of Algebra \\ Charles University \\ Sokolovsk\'a 83 \\ 186~75 Praha 8 \\ Czech Republic}
\email{andrew.kozlik@gmail.com}

\subjclass[2020]{05B07, 20N05}
\keywords{semisymmetric quasigroup, Mendelsohn triple system, Mendelsohn loop, nucleus, centre, Pasch configuration, Schreier extension}

\begin{document}
\begin{abstract}
A binary operation $\cdot$ which satisfies the identity $(x \cdot y) \cdot x = y$ is called a semisymmetric quasigroup.
We show that the nucleus of a semisymmetric quasigroup is either empty or an elementary abelian 2-group coinciding with the centre, and that a semisymmetric quasigroup with a non-empty nucleus is necessarily a Mendelsohn loop, i.e.\ the loop associated with a Mendelsohn triple system.
We derive necessary and sufficient conditions for the existence of a semisymmetric quasigroup of order~$n$ with nucleus of order~$m$.
Furthermore, we characterize the nuclear elements of a Mendelsohn loop in terms of a particular orientation of the Pasch configuration in the associated triple system.
\end{abstract}

\maketitle

\section{Introduction}

Let $V$ be a set of points.
A \emph{cyclic triple} on~$V$, denoted by $(x,y,z)$, or equivalently $(y,z,x)$ or $(z,x,y)$, is a set of ordered pairs $(x,y)$, $(y,z)$ and $(z,x)$ of distinct points from~$V$.
A \emph{Mendelsohn triple system} of order~$v$, MTS($v$), is a pair $(V,\mathcal{B})$ where $V$ is a set of $v$ points and $\mathcal{B}$ is a collection of cyclic triples of distinct points taken from~$V$ such that every ordered pair of distinct points from~$V$ appears in precisely one triple.
It is well known that an MTS($v$) exists if and only if $v\equiv 0$ or $1\pmod 3$ and $v\neq 6$, see~\cite{mendelsohn}.

Given an MTS $(V,\mathcal{B})$ one can define a binary operation~$\cdot$ on the set~$V$ by assigning $x\cdot x = x$ for all $x\in V$ and $x\cdot y=z$ for all $x$, $y \in V$, $x \neq y$, where $z$ is the third element in the cyclic triple containing the ordered pair $(x,y)$.
The resulting operation satisfies the identities $x\cdot x = x$ and $(x \cdot y) \cdot x = y$ for all $x$, $y\in V$.
The second of these identities is known as the \emph{semisymmetric law}.
Any binary operation satisfying these two identities is called an \emph{idempotent semisymmetric quasigroup} or, more simply, a \emph{Mendelsohn quasigroup}.
The process described above is reversible.
Given a Mendelsohn quasigroup one can obtain an MTS by assigning $(x,y,x\cdot y)\in\mathcal{B}$ for all $x$, $y\in V$, $x\neq y$.
There is therefore a one-to-one correspondence between Mendelsohn triple systems and Mendelsohn quasigroups.

A \emph{semisymmetric loop} or \emph{Mendelsohn loop} is a pair $(L,\cdot)$, where $L$ is a set containing an identity element~$e$ and $\cdot$ is an operation on $L$ satisfying the identities $x\cdot e = x$ and $(x \cdot y) \cdot x = y$.
It follows that $x\cdot x = (x\cdot e)\cdot x = e$ and that $e$ is a two-sided identity element, since $e\cdot x = (e\cdot x)\cdot e = x$.
If $(V,\mathcal{B})$ is an MTS($n$), then a Mendelsohn loop $(L,\cdot)$ is obtained by letting $L = V \cup \{e\}$ and defining $x\cdot x = e$ and $x \cdot y = z$ where $z$ is the third element in the cyclic triple containing the ordered pair $(x,y)$.
Again, the process is reversible and there is a one-to-one correspondence between Mendelsohn triple systems of order~$v$ and Mendelsohn loops of order~$v+1$.
Thus a Mendelsohn loop of order~$n$ exists if and only if $n\equiv 1$ or $2\pmod{3}$ and $n\neq 7$.

A \emph{Steiner loop} is a commutative Mendelsohn loop.
There is a similar one-to-one correspondence between Steiner loops and Steiner triple systems, see for example \cite[page 24]{CR}.
One particularly important class of Steiner loops are those that are isomorphic to $(\mathbb{F}_2^k,+)$ for some~$k$.
These correspond to the so-called \emph{projective} Steiner triple systems, see \cite[page 29]{CR}.

The binary operation will sometimes be replaced with juxtaposition, indicating higher precedence.
For example the semisymmetric law may then be written as $xy\cdot x = y$.

\begin{lemma}\label{lem:mirror}
Let $\cdot$ be a binary operation upon a set~$X$. If $xy \cdot x = y$ for all $x$, $y \in X$, then $x \cdot yx = y$ for all $x$, $y\in X$ as well.
\end{lemma}
\begin{proof}
We have $x \cdot yx = (yx \cdot y) \cdot yx = y$.
\end{proof}

From this lemma we can see that any semisymmetric binary operation on a set~$X$ necessarily yields a quasigroup in which $y\backslash x = x\cdot y = y/x$.

The \emph{left nucleus} $N_\lambda$, \emph{middle nucleus} $N_\mu$ and \emph{right nucleus} $N_\rho$ of a quasigroup~$Q$ are defined as
\begin{align*}
N_\lambda(Q) &= \{\, u\in Q : u \cdot xy = ux\cdot y\text{ for all } x,y\in Q\,\},\\
N_\mu(Q)     &= \{\, u\in Q : x \cdot uy = xu\cdot y\text{ for all } x,y\in Q\,\},\\
N_\rho(Q)    &= \{\, u\in Q : x \cdot yu = xy\cdot u\text{ for all } x,y\in Q\,\}.
\end{align*}
The \emph{nucleus} $N(Q) = N_\lambda(Q)\cap N_\mu(Q)\cap N_\rho(Q)$ of~$Q$ is a subquasigroup of~$Q$.
The \emph{centre} of $Q$ is defined as
\[
  Z(Q)=N(Q)\cap\{\, x\in Q : xy = yx \text{ for all } y\in Q\,\}.
\]

In the next section we study the nuclei of a semisymmetric quasigroup, finding that the left, middle and right nucleus always coincide with the centre, see Proposition~\ref{prop:nuclei}.
In Section~\ref{sec:existence} we determine the existence spectrum of semisymmetric quasigroups with empty nucleus, see Theorem~\ref{thm:empty}, and derive the necessary and sufficient conditions for the existence of a Mendelsohn loop of order~$n$ with nucleus of order~$m$, see Theorem~\ref{thm:nucleus}.
In Steiner loops the question of nucleus size was already resolved by Donovan et al.\ \cite{donovan1,donovan2}, who showed that the nucleus of a Steiner loop is an elementary abelian $2$-group and determined the complete spectrum of admissible nucleus orders for Steiner loops.
The present author revisited the topic in connection with the maxi-Pasch problem in Steiner triple systems~\cite{kozlik}.
Recently, Falcone, Figula and Galici~\cite{FFG} developed an extension theory for Steiner triple systems in terms of Schreier extensions of the associated Steiner loops.
Their approach builds on the fact that the non-trivial central elements of a Steiner loop are precisely the Veblen points of the triple system, i.e.\ points through which any two distinct triples generate a Fano plane, a characterization which goes back to Donovan~\cite{donovan1}, see also \cite[Theorem~2.4]{kozlik}.
The corresponding combinatorial characterization of the nuclear elements of a Mendelsohn loop is given in Proposition~\ref{prop:pasch}.
The results of Section~\ref{sec:existence} on Mendelsohn loops with non-trivial nucleus can be viewed as a non-commutative counterpart of the Schreier extension theory of Steiner loops.
One notable difference is that a Steiner loop whose centre has index at most~4 is an elementary abelian group~\cite{donovan2}, whereas Example~\ref{ex:8} exhibits a non-associative Mendelsohn loop whose nucleus has index~4.

\section{Properties}

\begin{lemma}\label{lem:commut}
Let $Q$ be a semisymmetric quasigroup.
If $u$ lies in the left, middle or right nucleus of~$Q$, then $ux = xu$ for all $x\in Q$.
\end{lemma}
\begin{proof}
We will show that if $u$ lies in one of the nuclei, then $u \cdot ux = x$ or $ xu \cdot u = x$ for all $x\in Q$.
In a semisymmetric quasigroup either one of these equalities implies that $ux = xu$.

If $u\in N_\lambda(Q)$, then $u\cdot ux = uu \cdot x = (u\cdot (xu \cdot x)) \cdot x = ((u \cdot xu)\cdot x) \cdot x = xx \cdot x = x$.

If $u\in N_\mu(Q)$, then $u\cdot ux = uu \cdot x = uu \cdot (u \cdot xu) = (uu \cdot u) \cdot xu = u \cdot xu = x$.

If $u\in N_\rho(Q)$, then $xu \cdot u = x \cdot uu = x \cdot ((x \cdot ux) \cdot u) = x\cdot (x \cdot (ux \cdot u)) = x \cdot xx = x$.
\end{proof}

\begin{proposition}
\label{prop:nuclei}
Let $Q$ be a semisymmetric quasigroup.
Then $Z(Q) = N(Q) = N_\lambda(Q) = N_\mu(Q) = N_\rho(Q)$ and if $N(Q)$ is non-empty, then it is an elementary abelian 2-group.
\end{proposition}
\begin{proof}
We show that $N_\lambda(Q)\subseteq N_\mu(Q)\subseteq N_\rho(Q)\subseteq N_\lambda(Q)$.

If $u\in N_\lambda(Q)$, then for any $x$, $y\in Q$ we have
\[
  xu \cdot y = ux \cdot y = u \cdot xy = x\cdot ((u \cdot xy) \cdot x) = x\cdot (u \cdot (xy \cdot x)) = x \cdot uy.
\]

If $u\in N_\mu(Q)$, then for any $x$, $y\in Q$ we have
\[
\begin{split}
  xy \cdot u = u \cdot xy &= (uy \cdot (u \cdot xy)) \cdot uy \\
                          &= ((uy \cdot u) \cdot xy) \cdot uy = (y \cdot xy) \cdot uy = x \cdot uy = x \cdot yu.
\end{split}
\]

If $u\in N_\rho(Q)$, then for any $x$, $y\in Q$ we have
\[
  u \cdot xy = xy \cdot u = (y \cdot (xy \cdot u)) \cdot y = ((y \cdot xy) \cdot u) \cdot y = xu \cdot y = ux \cdot y.
\]

By the previous lemma the nucleus coincides with the centre.
The rest follows from the facts that every non-empty associative quasigroup is a group and that $x\cdot x$ is the identity element of~$N(Q)$ for every $x\in N(Q)$.
\end{proof}

\begin{corollary}
A Mendelsohn loop is associative if and only if it is isomorphic to $(\mathbb{F}_2^k,+)$ for some~$k$.
\end{corollary}

For $a\in Q$ denote by $L_a$ the \emph{left translation} $x\mapsto a\cdot x$.
A loop~$Q$ is said to be \emph{left conjugacy closed} (LCC) if for all $x$, $y\in Q$ there exists $w\in Q$ such that $L_x L_y L_x^{-1} = L_w$, i.e.\ $x \cdot (y \cdot (x\bs z)) = wz$ for all $z\in Q$.
If such $w$ exists, then by substituting $z = x$ we see that it must be $w = (xy)/x$.

\begin{lemma}\label{lem:lcc:asoc}
A Mendelsohn loop is LCC if and only if it is associative.
\end{lemma}
\begin{proof}
If $Q$ is an LCC loop, then we have $x \cdot (y \cdot (x\bs z)) = ((xy)/x) \cdot z$ for all $x$, $y$, $z\in Q$.
In a Mendelsohn loop this expression can be rewritten as $x \cdot (y \cdot zx) = (x \cdot xy) \cdot z$.
For $z=e$ it then simplifies to $y = x \cdot xy$.
Now for all $x$, $y$, $z\in Q$ we have $y \cdot zx = (x \cdot (y \cdot zx)) \cdot x = ((x \cdot xy) \cdot z) \cdot x = yz \cdot x$.

Conversely, if $Q$ is associative, then it is a group and is LCC with $w = xyx^{-1}$.
\end{proof}

The nuclear elements admit the following combinatorial characterization, a directed counterpart of the characterization of the central elements of a Steiner loop as the Veblen points of the corresponding Steiner triple system.

\begin{proposition}\label{prop:pasch}
Let $(V,\mathcal{B})$ be a Mendelsohn triple system, let $L$ be the associated Mendelsohn loop and let $u\in V$.
Then $u \in N(L)$ if and only if for any two distinct elements $x$, $y \in V\setminus\{u\}$ either $\{(x,u,y)$, $(u,x,y)\}\subseteq\mathcal{B}$ or there exist elements $v$, $w$, $z\in V$ such that $\{(x,u,v)$, $(u,y,w)$, $(v,y,z)$, $(x,w,z)\}\subseteq\mathcal{B}$.
\end{proposition}

\begin{proof}
Suppose that $u \in N(L)$ and let $x$, $y \in V\setminus\{u\}$ be distinct.
Then there exists $(x,u,v)\in\mathcal{B}$ for some $v\in V$.
If $v = y$, then $xu = y$ and by Lemma~\ref{lem:commut} we have $ux = y$, giving us $(u,x,y)\in\mathcal{B}$, which is the first case.
Suppose now that $v \neq y$.
There also exists $(u,y,w)$, $(v,y,z)\in\mathcal{B}$ for some $w$, $z\in V$, where $w\neq x$, because $w = x$ would imply $v = y$.
It follows that $xw = x\cdot uy = xu \cdot y = vy = z$, so $(x,w,z)\in\mathcal{B}$.

Conversely, suppose that for any two distinct elements $x$, $y\in V\setminus\{u\}$ at least one of the two alternatives holds.
First let $x \in V \setminus \{u\}$ and $y = xu$.
The second alternative cannot occur for this pair, since $(x,u,v)\in\mathcal{B}$ gives $v = xu = y$, duplicating a point of $(v,y,z)$.
Hence $\{(x,u,y)$, $(u,x,y)\}\subseteq\mathcal{B}$ and $y = ux$.
Thus $u$ commutes with all elements of~$L$.
Furthermore $xu\cdot y = yy = e = xx = x\cdot uy$.
Now let $x$, $y \in L$ be arbitrary.
For $x = e$ or $y = e$, trivially $xu\cdot y = x\cdot uy$.
If $x = y$, then using Lemma~\ref{lem:mirror} we have $xu\cdot y = xu \cdot x = u = x \cdot ux = x\cdot uy$.
If $x \neq y$ and $x = u$, then $xu\cdot y = uu \cdot y = y = u\cdot yu = u \cdot uy = x\cdot uy$.
Similarly for $x \neq y$ and $y = u$.
Finally, if $x$ and $y$ are distinct elements of $V \setminus \{u\}$ with $y \neq xu$, then the first alternative is ruled out, since $(x,u,y)\in\mathcal{B}$ would give $y = xu$, so the second alternative applies and we have $xu\cdot y = vy = z = xw = x \cdot uy$.
Hence $u \in N_\mu(L)$, and $u \in N(L)$ by Proposition~\ref{prop:nuclei}.
\end{proof}

Note that the six points of the configuration in the second alternative are necessarily distinct, since $u = z$ would give $x = uv = zv = y$, while $v = w$ would give $uy = w = v = xu = ux$, hence $x = y$, as $u$ commutes with all elements of~$L$.
Disregarding the cyclic orders, the four triples of the configuration therefore form a Pasch configuration, with each of the six points lying on exactly two of the four blocks.
The cyclic orders constitute one of the orientations of the Pasch configuration into cyclic triples.
It is easy to verify that up to isomorphism there are exactly three such orientations:
\begin{enumerate}
\item $\{(x,u,v)$, $(u,w,y)$, $(v,y,z)$, $(x,z,w)\}$,
\item $\{(x,u,v)$, $(u,y,w)$, $(v,y,z)$, $(x,w,z)\}$ and
\item $\{(x,u,v)$, $(u,y,w)$, $(v,z,y)$, $(x,w,z)\}$
\end{enumerate}
whose automorphism groups are $A_4$, $C_4$ and $C_3$, respectively.
The automorphism group of the first configuration acts transitively on the six points and is generated by the permutations $(x,u,v)(y,z,w)$ and $(v,w)(x,y)$.
The configuration of Proposition~\ref{prop:pasch} is the second of the three, and its automorphism group is generated by $(u,x,z,y)(v,w)$, while the automorphism group of the third configuration is generated by $(x,u,v)(y,z,w)$.

If $L$ is commutative, i.e.\ a Steiner loop, then every cyclic triple occurs together with its reverse.
The first alternative in Proposition~\ref{prop:pasch} then states that $x$ and $y$ lie in a common triple with~$u$, and the second reduces to the Pasch configuration generated by two distinct triples through a Veblen point.

If $u\in V$ is a nuclear element of a Mendelsohn loop of order~$n$, then of the $(n-2)(n-3)$ ordered pairs of distinct elements $x$, $y\in V\setminus\{u\}$ exactly $n-2$ satisfy $y = xu$, and each of the remaining $(n-2)(n-4)$ pairs determines an oriented Pasch configuration as in Proposition~\ref{prop:pasch}, with distinct pairs determining distinct configurations, since no non-trivial automorphism of the configuration fixes~$u$.
Thus any two cyclic triples through a nuclear point either share their point set or extend to an oriented Pasch configuration.
This is a directed counterpart of the fact that the number of Pasch configurations through a fixed Veblen point of a Steiner triple system of order~$v$ is $(v-1)(v-3)/4$, see \cite[Lemma~3.11]{FFG}.
In the commutative case exactly four of the sixteen orientations of a Pasch configuration through a Veblen point are of the type in Proposition~\ref{prop:pasch} with $u$ in the position indicated there, namely the one displayed there and those obtained from it by reversing the two cyclic triples through~$u$, the two not through~$u$, or all four.
This explains the factor of four between the two counts.

\section{Existence}
\label{sec:existence}

\begin{lemma}\label{lem:loop}
Let $Q$ be a semisymmetric quasigroup.
If the nucleus of~$Q$ is non-empty, then $Q$ is a loop.
\end{lemma}
\begin{proof}
If $N(Q)$ is non-empty, then by Proposition~\ref{prop:nuclei} it is a group with an identity element~$e$.
We show that $e$ is in fact an identity element of the whole quasigroup~$Q$.
For any $x\in Q$ we have $ex = (e\cdot ex)\cdot e = (ee\cdot x)\cdot e = ex\cdot e = x$, where the second equality holds because $e\in N_\lambda(Q)$.
Using Lemma~\ref{lem:commut} we have $xe = ex = x$.
\end{proof}

\begin{theorem}\label{thm:empty}
A semisymmetric quasigroup of order~$n$ with an empty nucleus exists for all $n \geq 3$.
\end{theorem}
\begin{proof}
For $n\geq 3$ let $Q = (\mathbb{Z}_n, \cdot)$ where $x\cdot y = -x - y$.
Then $(x\cdot y)\cdot x = y$, so $Q$ is a semisymmetric quasigroup.
If $Q$ were a loop with identity element~$e$, then $x\cdot x = (x\cdot e)\cdot x = e$ for every $x\in Q$, which is impossible, since $0\cdot 0 = 0$ and $1\cdot 1 = -2$ are distinct for $n\geq 3$.
Thus $N(Q)$ is empty by Lemma~\ref{lem:loop}.
\end{proof}

The quasigroup constructed in the proof of Theorem~\ref{thm:empty} is commutative, but non-commutative semisymmetric quasigroups with empty nucleus exist in abundance.
The nucleus of every non-trivial Mendelsohn quasigroup~$Q$ is also empty, since idempotency yields $uu = u \neq v = vv$ for any pair of distinct $u$, $v\in Q$.

In the remainder of this paper we will focus on semisymmetric quasigroups with a non-empty nucleus, i.e.\ Mendelsohn loops.
Khatirinejad et al.~\cite{OstergardMTS} have shown that the number of isomorphism classes of MTS($v$) is $1$, $1$, $0$, $4$, $20$, $241$, $9\,801\,188$ and $13\,710\,290\,116$ for $v = 3$, $4$, $6$, $7$, $9$, $10$, $12$ and $13$.
The overwhelming majority of the corresponding Mendelsohn loops have a trivial nucleus $N(Q)=\{e\}$.

With the help of a computer running the model builder Mace4~\cite{mace}, we can obtain a census of the nuclei of Mendelsohn loops of order up to~16.
Firstly, there are the projective Steiner loops of orders 2, 4, 8 and 16, where the nucleus is the entire loop.
Up to isomorphism there is one Mendelsohn loop of order 8 with nucleus of order~2, see Example~\ref{ex:8},
one Mendelsohn loop of order~10 with nucleus of order~2,
one of order 16 with nucleus of order~4 and
85 Mendelsohn loops of order 16 with nucleus of order~2.
The loop of order~10 is the direct product of the unique Mendelsohn loop of order~5 with $(\mathbb{F}_2,+)$ and the loop of order~16 with nucleus of order~4 is the direct product of the loop from Example~\ref{ex:8} with $(\mathbb{F}_2,+)$.
Indeed, the direct product of two Mendelsohn loops is a Mendelsohn loop and its nucleus is the direct product of their nuclei, as one verifies componentwise, so in each case the product has a nucleus of the required order and the identification follows from the uniqueness.
One of the 85 loops of order 16 with nucleus of order~2 is a Steiner loop corresponding to the Steiner triple system of order~15 with automorphism group of order~192, i.e.\ System \#~2 in~\cite{CR}.
All remaining Mendelsohn loops of order less than or equal to~16 have trivial nucleus.

\begin{proposition}
\label{prop:triv}
A Mendelsohn loop of order~$n$ with trivial nucleus exists if and only if $n\equiv 1$ or $2\pmod{3}$ and $n\not\in\{2,4,7\}$.
\end{proposition}
\begin{proof}
A loop~$L$ is said to be \emph{anticommutative} if $xy=yx$ implies $x=y$ for any $x$, $y\in L\setminus\{e\}$.
By Lemma~\ref{lem:commut}, every anticommutative Mendelsohn loop of order different from~$2$ has trivial nucleus.
By \cite{puremts} an anticommutative Mendelsohn loop of order~$n$ exists if and only if $n\equiv 1$ or $2\pmod{3}$ and $n\not\in \{2,4,7\}$.
A Mendelsohn loop of order~$7$ does not exist and the unique Mendelsohn loops of orders~$2$ and~$4$ are associative.
\end{proof}

\begin{lemma}
Let $L$ be a Mendelsohn loop with nucleus of order~$n$ and let $x\in L\setminus N(L)$.
Then the result of adjoining $x$ to $N(L)$ is the subloop $N(L)\cup xN(L)$ of order~$2n$.
\end{lemma}
\begin{proof}
First, let us verify that the set $N(L)\cup xN(L)$ is closed under the binary operation.
For any $u$, $v\in N(L)$ we have $u\cdot xv = xv\cdot u = x\cdot vu \in xN(L)$ and $xu\cdot xv = (xu\cdot x)\cdot v = uv \in N(L)$.

The order of $xN(L)$ is~$n$ and $N(L)\cap xN(L) = \emptyset$.
If the intersection were not empty, there would exist $u$, $v\in N(L)$, such that $u=xv$, implying that $x\in N(L)$, which is a contradiction.
\end{proof}

\begin{lemma}\label{lem:nonex}
There exists no Mendelsohn loop of order~$n$ with nucleus of order~$n/2$.
\end{lemma}
\begin{proof}
Assume to the contrary that there exists a Mendelsohn loop~$L$ of order~$n$ with nucleus of order~$n/2$ and let $x\in L\setminus N(L)$.
Then by the previous lemma $L = N(L)\cup xN(L)$ and for any $z\in L\setminus N(L)$ there exists $u\in N(L)$ such that $z = xu$.
Now for any $y\in L$ we have
\[
  xy \cdot z = xy \cdot xu = (xy\cdot x) \cdot u = (x\cdot yx) \cdot u = x\cdot (yx\cdot u) = x\cdot (y\cdot xu) = x\cdot yz,
\]
proving that $x\in N_\lambda(L)$, which is a contradiction.
\end{proof}

The validity of the previous lemma can also be seen from Lemma~\ref{lem:lcc:asoc} and the fact that a loop with nucleus of index two is conjugacy closed, see \cite{goodaire} or~\cite{drapal}.

\begin{example}\label{ex:8}
The following Cayley table shows a Mendelsohn loop of order~$8$ with nucleus~$\{0,4\}$.
\begin{center}
\begin{tabular}{cccccccc}
0 & 1 & 2 & 3 & 4 & 5 & 6 & 7\\
1 & 0 & 7 & 2 & 5 & 4 & 3 & 6\\
2 & 3 & 0 & 5 & 6 & 7 & 4 & 1\\
3 & 6 & 1 & 0 & 7 & 2 & 5 & 4\\
4 & 5 & 6 & 7 & 0 & 1 & 2 & 3\\
5 & 4 & 3 & 6 & 1 & 0 & 7 & 2\\
6 & 7 & 4 & 1 & 2 & 3 & 0 & 5\\
7 & 2 & 5 & 4 & 3 & 6 & 1 & 0\\
\end{tabular}
\end{center}
This loop is used in the proof of Theorem~\ref{thm:nucleus} to handle the case $n/m = 4$.
\end{example}

A subloop~$K$ of~$L$ is said to be \emph{normal} in~$L$ if $xK=Kx$, $x(yK)=(xy)K$ and $(xK)y=x(Ky)$ for all $x$, $y\in L$.
The \emph{factor loop} $L/K$ is then defined in the usual way.
It follows from Lemma~\ref{lem:commut} and the defining identities of the nuclei that if $L$ is a Mendelsohn loop, then $N(L)$ is normal in~$L$.

\begin{theorem}\label{thm:nucleus}
Let $n = 2^k\eta$ be a positive integer, where $\eta$ is odd.
A Mendelsohn loop of order~$n$ with nucleus of order~$m$ exists if and only if $n\equiv 1$ or $2\pmod{3}$, $n\neq 7m$, and
\begin{enumerate}
\item[(a)] $\eta = 1$, $(n,m)\neq (4,1)$ and $m=2^i$, where $i\in\{0,1,\dots,k-2\}\cup\{k\}$, or
\item[(b)] $\eta > 1$ and $m=2^i$, where $i\in\{0,1,\dots,k\}$.
\end{enumerate}
\end{theorem}
\begin{proof}
Let $L$ be a Mendelsohn loop.
Since the factor loop $L/N(L)$ is also a Mendelsohn loop, we have $n/m\neq 7$.
By Proposition~\ref{prop:nuclei} the nucleus $N(L)$ is an abelian 2-group and thus is isomorphic to $(\mathbb{F}_2^i,+)$ for some~$i$.
The necessity of the remaining conditions follows from Lemma~\ref{lem:nonex} and from the fact that the unique Mendelsohn loop of order~$4$ is associative and hence has nucleus equal to the entire loop.

Let $n$ and~$m=2^i$ be integers satisfying the conditions given in the statement of the theorem.
Then $n/m\equiv 1$ or $2\pmod{3}$ and $n/m \not\in \{2,7\}$.
If $n/m\neq 4$, then by Proposition~\ref{prop:triv} there exists a Mendelsohn loop~$K$ of order~$n/m$ with trivial nucleus.
The loop~$K\times\mathbb{F}_2^i$ of order~$n$ satisfies the semisymmetric identity and has nucleus~$\{e\}\times\mathbb{F}_2^i$ of order~$m$.
If $n/m=4$, then instead let $K$ be the Mendelsohn loop of order~$8$ with nucleus of order~$2$ that is given in Example~\ref{ex:8}.
The loop~$K\times\mathbb{F}_2^{i-1}$ is a Mendelsohn loop of order~$n$ with nucleus~$N(K)\times\mathbb{F}_2^{i-1}$ of order~$m$.
\end{proof}

\begin{corollary}
A Mendelsohn loop of order~$n$ with non-trivial nucleus exists if and only if $n\equiv 2$ or $4\pmod{6}$ and $n\neq 14$.
\end{corollary}

It is interesting to note that the loop given in Example~\ref{ex:8} is a special case of a more general construction of Mendelsohn loops with non-trivial nucleus.

\begin{proposition}
\label{prop:phi}
Let $L$ be a Mendelsohn loop and $\varphi: L \times L\to \mathbb{F}_2^i$.
Define a binary operation $*_\varphi$ on the set $L\times\mathbb{F}_2^i$ as
\[
  (x,u) *_\varphi (y,v) = (x y,\, u + v + \varphi(x,y)).
\]
Then $L_\varphi = (L\times\mathbb{F}_2^i,*_\varphi)$ is a Mendelsohn loop if and only if
\begin{equation}\label{eq:phicond}
\tag{$\Phi$}
\varphi(x,y) = \varphi(x y, x) \quad\text{and}\quad \varphi(x,e) = \varphi(e,y) \quad\text{for all $x$, $y\in L$.}
\end{equation}
\end{proposition}
\begin{proof}
For $*_\varphi$ the semisymmetric law $((x,u) *_\varphi (y,v)) *_\varphi (x,u) = (y,v)$ can be rewritten as
\[
  (xy \cdot x,\, v + \varphi(x,y) + \varphi(xy,x)) = (y,v),
\]
which is equivalent to the condition that $\varphi(x, y) = \varphi(xy, x)$ for all $x,y\in L$.

Assume that $L_\varphi$ is a loop and $(g,h)$ is its identity element.
Then $(x,u) = (x,u) *_\varphi (g,h) = (xg,\, u + h + \varphi(x,g))$, thus $g$ is the identity element of~$L$ and $\varphi(x,g) = h$ for all $x\in L$.
Similarly for $(g,h) *_\varphi (x,u)$ we find that $\varphi(g,x) = h$ for all $x\in L$.
Thus $\varphi(x,e) = \varphi(e,y)$ for all $x$, $y\in L$.

If there exists $h\in \mathbb{F}_2^i$ such that $\varphi(x,e) = h = \varphi(e,x)$ for all $x\in L$, then $(e,h)$ is an identity element in~$L_\varphi$, as one can easily verify.
\end{proof}
If $\varphi$ satisfies conditions~\eqref{eq:phicond} with $\varphi(x,e) = h = \varphi(e,y)$, then the mapping $(x,u)\mapsto(x,u+h)$ is an isomorphism of~$L_\varphi$ onto~$L_{\varphi'}$, where $\varphi'(x,y) = \varphi(x,y)+h$ again satisfies conditions~\eqref{eq:phicond} and $\varphi'(x,e) = 0 = \varphi'(e,y)$.
There is therefore no loss of generality in assuming that $\varphi(x,e) = 0 = \varphi(e,y)$.
For $h = 0$, the mapping~$\varphi$ is a factor system of the extension~$L_\varphi$ in the terminology of the theory of Schreier extensions, cf.~\cite{FFG,strambachstuhl}, and conditions~\eqref{eq:phicond} single out the factor systems for which the extension is semisymmetric.

Let $(V,\mathcal{B})$ be the Mendelsohn triple system associated with~$L$.
Each cyclic triple in~$\mathcal{B}$ is of the form $(x,y,xy)$ for some $x$, $y\in V$ and contains the ordered pairs $(x,y)$, $(y,xy)$ and $(xy,x)$.
The condition $\varphi(x,y) = \varphi(xy,x)$ thus states that $\varphi$ is constant on the ordered pairs of each cyclic triple of~$\mathcal{B}$.
For the remaining ordered pairs $(x,y)\in L\times L$, those with $e\in\{x,y\}$ or $x = y$, the two conditions of~\eqref{eq:phicond} force $\varphi$ to take on a single value~$h$, since $\varphi(x,x) = \varphi(xx,x) = \varphi(e,x) = h$.
Consequently, under the normalization $h = 0$ assumed above, the mappings $\varphi$ satisfying conditions~\eqref{eq:phicond} are precisely those obtained from an arbitrary mapping $f:\mathcal{B}\to\mathbb{F}_2^i$ by setting $\varphi(x,y) = f((x,y,xy))$ for distinct $x$, $y\in V$ and $\varphi(x,y) = 0$ otherwise.

\begin{proposition}\label{prop:nucchar}
Let $L$ be a Mendelsohn loop and let $\varphi: L\times L\to\mathbb{F}_2^i$ satisfy conditions~\eqref{eq:phicond}.
Then $(x,u)\in N(L_\varphi)$ if and only if $x\in N(L)$ and
\[
  \varphi(x,y) + \varphi(xy,z) = \varphi(y,z) + \varphi(x,yz)
  \quad\text{for all $y$, $z\in L$.}
\]
In particular $\{e\}\times\mathbb{F}_2^i\subseteq N(L_\varphi)$, and if $N(L)$ is trivial, then $N(L_\varphi) = \{e\}\times\mathbb{F}_2^i$.
\end{proposition}
\begin{proof}
By Proposition~\ref{prop:nuclei} we have $(x,u)\in N(L_\varphi)$ if and only if $(x,u)\in N_\lambda(L_\varphi)$.
Comparing the two coordinates of $((x,u) *_\varphi (y,v)) *_\varphi (z,w)$ and $(x,u) *_\varphi ((y,v) *_\varphi (z,w))$ shows that this holds if and only if $xy\cdot z = x\cdot yz$ and $\varphi(x,y) + \varphi(xy,z) = \varphi(y,z) + \varphi(x,yz)$ for all $y$, $z\in L$.
The former condition states that $x\in N_\lambda(L) = N(L)$.
For $x = e$ the latter condition reduces to $\varphi(e,y) = \varphi(e,yz)$, which holds by~\eqref{eq:phicond}.
\end{proof}

Since the criterion in Proposition~\ref{prop:nucchar} does not involve~$u$, the nucleus of~$L_\varphi$ is a union of cosets of $\{e\}\times\mathbb{F}_2^i$.
In the commutative case of Steiner loops this criterion appears in~\cite[Remark~5.1]{FFG}, where it characterizes the Veblen points of the extension.

\begin{proposition}\label{prop:nucexact}
Let $L$ be a Mendelsohn loop of order at least~4, $h\in\mathbb{F}_2^i$, and $\varphi: L \times L\to \mathbb{F}_2^i$ a mapping such that for all distinct $x$, $y\in L\setminus\{e\}$ with $xy = yx$,
\[
  \varphi(x,y) = \varphi(xy,x) = \varphi(y,xy) \neq \varphi(y,x) = \varphi(yx,y) = \varphi(x,yx),
\]
and $\varphi(x,y) = h$ otherwise, i.e.\ if $xy \neq yx$, $e\in\{x,y\}$ or $x = y$.
Then $L_\varphi$ is a Mendelsohn loop with nucleus $\{e\}\times\mathbb{F}_2^i$.
\end{proposition}
\begin{proof}
If $x$ and~$y$ are distinct commuting elements of $L\setminus\{e\}$, then the Mendelsohn triple system associated with~$L$ contains the cyclic triple $(x,y,xy)$ together with its reverse $(y,x,xy)$, and the six ordered pairs appearing in the displayed condition are precisely the ordered pairs of these two triples.
The condition thus states that $\varphi$ is constant on the ordered pairs of each such cyclic triple and attains the single value~$h$ on all remaining ordered pairs of $L\times L$, which amounts to conditions~\eqref{eq:phicond} by the discussion following Proposition~\ref{prop:phi}.
Therefore by Proposition~\ref{prop:phi} the loop $L_\varphi$ is a Mendelsohn loop and by Proposition~\ref{prop:nucchar} its nucleus contains $\{e\}\times\mathbb{F}_2^i$.

To see that the other elements do not lie in the nucleus we can use Lemma~\ref{lem:commut} and show that for each $(x,u)\in L\times\mathbb{F}_2^i$ such that $x\neq e$, there exists $(y,v)\in L\times\mathbb{F}_2^i$ which does not commute with~$(x,u)$.
Indeed for any $y\in L\setminus\{x,e\}$ and $v\in\mathbb{F}_2^i$ we have
\[
  (x,u) *_\varphi (y,v) = (xy,\, u + v + \varphi(x,y)) \neq (yx,\, u + v + \varphi(y,x)) = (y,v) *_\varphi (x,u),
\]
because either $xy\neq yx$ or $\varphi(x,y)\neq\varphi(y,x)$.
\end{proof}

Example~\ref{ex:8} demonstrates the construction from the previous proposition by starting with $(L,\cdot) = (\mathbb{Z}_4,\oplus)$, where $\oplus$ is the bitwise exclusive-or operation, and defining $\varphi$ as $\varphi(1, 2) = \varphi(2, 3) = \varphi(3, 1) = 1$ and as~$0$ elsewhere on its domain.
Each element $(x,u)\in\mathbb{Z}_4\times\mathbb{F}_2$ of the resulting loop is represented as~$x+4u\in\mathbb{Z}_8$ in the example.

For a Steiner loop~$L$ and $i = 1$, the mappings~$\varphi$ satisfying the hypotheses of the previous proposition with $h = 0$ correspond exactly to the orientations of the underlying Steiner triple system, and the resulting loops~$L_\varphi$ are the oriented Steiner loops of exponent~2 introduced by Strambach and Stuhl~\cite{strambachstuhl}.
In particular, the loop of Example~\ref{ex:8} is the oriented Steiner loop of exponent~2 associated with the oriented STS(3).

Every Mendelsohn loop with a non-trivial nucleus arises from the construction of Proposition~\ref{prop:phi} via a suitable choice of~$\varphi$, as we show in Proposition~\ref{prop:complete} below.
This is a semisymmetric analogue of the classical fact that every loop containing a given abelian group in its centre is a factor system extension of that group by the corresponding quotient loop~\cite[p.~334]{bruck}.

\begin{proposition}\label{prop:complete}
Let $L$ be a Mendelsohn loop with nucleus isomorphic to~$\mathbb{F}_2^i$.
For $K = L/N(L)$ there exists $\varphi: K\times K\to\mathbb{F}_2^i$ satisfying
conditions~\eqref{eq:phicond} such that $L\cong K_\varphi$.
\end{proposition}
\begin{proof}
Choose an isomorphism $\theta: \mathbb{F}_2^i \to N(L)$ and a section $s: K\to L$ with $s(e_K) = e_L$.
Since for any $x$, $y\in K$ the elements $s(x)\cdot s(y)$ and $s(xy)$ have the same image $xy$ under the canonical projection $L \to K$ and $K$ has exponent~2, the product of the two elements lies in~$N(L)$.
This allows us to define $\varphi: K\times K\to\mathbb{F}_2^i$ by $\varphi(x,y) = \theta^{-1}\bigl(s(x)\,s(y)\cdot s(xy)\bigr)$.
Equivalently, $s(xy)\, \theta\bigl(\varphi(x,y)\bigr) = s(x)\,s(y)$ by applying the semisymmetric law to the definition.
Whenever $s(x)\,s(y)\cdot s(xy) = z$, it follows that $s(y) = s(xy)\,z \cdot s(x) = z \cdot s(xy)\,s(x)$, hence $s(xy)\,s(x)\cdot s(y) = z$, using the semisymmetric law and the centrality of~$z$.
Thus we have
\[
  \varphi(x,y) = \theta^{-1}\bigl(s(xy)\,s(x)\cdot s(y)\bigr) = \theta^{-1}\bigl(s(xy)\,s(x)\cdot s(xy\cdot x)\bigr) = \varphi(xy,x),
\]
and the condition $\varphi(x,e_K) = 0 = \varphi(e_K,y)$ follows from $s(e_K)=e_L$.

Since every element of~$L$ decomposes uniquely as a product of a coset representative and an element of~$N(L)$, the map $\psi: K_\varphi \to L$ defined by $\psi(x,u) = s(x)\cdot \theta(u)$ is a bijection, and
\[
\begin{split}
  \psi(x,u)\cdot\psi(y,v)
  &= s(x) \cdot s(y) \cdot \theta(u) \cdot \theta(v)
   = s(xy) \cdot \theta\bigl(\varphi(x,y)\bigr) \cdot \theta(u) \cdot \theta(v) \\
  &= s(xy) \cdot \theta\bigl(\varphi(x,y) + u + v\bigr)
   = \psi\bigl((x,u)*_\varphi(y,v)\bigr),
\end{split}
\]
where the first equality uses the centrality of~$N(L)$ and the second the identity noted above.
\end{proof}

\section{Open problems}
Proposition~\ref{prop:complete} describes every Mendelsohn loop with a non-trivial nucleus as~$L_\varphi$ for a suitable choice of~$\varphi$, but the correspondence between the mappings~$\varphi$ and the resulting loops is not one-to-one.
For Steiner loops, the corresponding equivalence and isomorphism theory of Schreier extensions was developed in~\cite[Section~5]{FFG}.
It is an open problem to extend this theory to Mendelsohn loops and to use it to determine the isomorphism classes of Mendelsohn loops with a given nucleus beyond the orders accessible to the exhaustive search of Section~\ref{sec:existence}, as was done in~\cite{FG} for the Steiner triple systems of orders 19 and~27 that contain Veblen points.

The oriented Pasch configurations also deserve further study.
Reading the three orientations with the various points of the configuration in the role of~$u$ yields several conditions on an element of a Mendelsohn loop, each of which can be expressed as an identity on the left translations of the loop.
Proposition~\ref{prop:pasch} identifies a condition arising from the second orientation as membership in the nucleus, and it is an open problem to determine which of the remaining conditions define classes with non-trivial structure in Mendelsohn loops.

\section{Acknowledgements}
The author would like to thank Ale\v s Dr\' apal for pointing out the possible connection of Lemma~\ref{lem:nonex} with LCC loops and for suggesting the investigation of the binary operation in Proposition~\ref{prop:phi}, and Kyle M. Lewis for suggesting the construction that is used in the proof of
Theorem~\ref{thm:empty}.
The author acknowledges the use of Claude Fable 5 (Anthropic) in developing the statements and proofs of Propositions \ref{prop:nucchar} and~\ref{prop:complete}, the discussion following Proposition~\ref{prop:pasch} and in conducting the literature review, which brought the papers \cite{FFG}, \cite{FG} and \cite{strambachstuhl} to the author's attention.
The author has verified all resulting mathematical content and takes full responsibility for its correctness.

\end{document}